\documentclass{article}
\usepackage{amssymb}
\usepackage{amsmath}
\usepackage{amsthm}
\usepackage{epsfig}
\usepackage{colonequals}
\usepackage{enumerate}
\usepackage{hyperref}
\usepackage{float}

\newtheorem{theorem}{Theorem}
\newtheorem{proposition}{Proposition}
\newtheorem{corollary}[theorem]{Corollary}
\newtheorem{lemma}[theorem]{Lemma}

\newtheorem{problem}[theorem]{Problem}
\newcommand{\eA}{\mathcal{A}}
\newcommand{\eB}{\mathcal{B}}
\newcommand{\FF}{\mathcal{F}}
\newcommand{\shm}{\,\triangledown\,}
\newcommand{\legendre}[2]{\genfrac{(}{)}{}{}{#1}{#2}}

\title{$1$-subdivisions, the fractional chromatic number and the Hall ratio}
\author{%
     Zden\v{e}k Dvo\v{r}\'ak\thanks{Computer Science Institute (CSI) of Charles University,
           Malostransk{\'e} n{\'a}m{\v e}st{\'\i} 25, 118 00 Prague, 
           Czech Republic. E-mail: \protect\href{mailto:rakdver@iuuk.mff.cuni.cz}{\protect\nolinkurl{rakdver@iuuk.mff.cuni.cz}}.
           Supported by the project 17-04611S (Ramsey-like aspects of graph
	   coloring) of Czech Science Foundation.}\and
     Patrice Ossona de Mendez\thanks{Centre d'Analyse et de Math\'ematiques Sociales (CNRS, UMR 8557), Paris, France
	and 
Computer Science Institute of Charles University (IUUK), Prague, Czech Republic
    and 
 Department of Mathematics, Zhejiang Normal University, Jinhua, China.  
 E-mail: \protect\href{mailto:pom@ehess.fr}{\protect\nolinkurl{pom@ehess.fr}}.
                Supported by grant ERCCZ LL-1201 and by the European Associated Laboratory ``Structures in
Combinatorics'' (LEA STRUCO).}\and
     Hehui Wu\thanks{Shanghai Center for Mathematical Sciences, Shanghai, China. E-mail:\protect\href{mailto:hhwu@fudan.edu.cn}{\protect\nolinkurl{hhwu@fudan.edu.cn}}
     }
}
\date{\today}
\begin{document}
\maketitle

\begin{abstract}
The \emph{Hall ratio} of a graph $G$ is the maximum of $|V(H)|/\alpha(H)$
over all subgraphs $H$ of $G$.  It is easy to see that the Hall ratio of a graph is a lower
bound for the fractional chromatic number.  It has been asked whether conversely,
the fractional chromatic number is upper bounded by a function of the Hall ratio.
We answer this question in negative, by showing two results of independent
interest regarding $1$-subdivisions (the \emph{$1$-subdivision} of a graph is
obtained by subdividing each edge exactly once).
\begin{itemize}
\item For every $c>0$, every graph of sufficiently large average
degree contains as a subgraph the $1$-subdivision of a graph of fractional chromatic number
at least $c$.
\item For every $d>0$, there exists a graph $G$ of average degree at least $d$
such that every graph whose $1$-subdivision appears as a subgraph of $G$ has
Hall ratio at most $18$.
\end{itemize}
We also discuss the consequences of these results in the context of graph classes with
bounded expansion.
\end{abstract}

\section{Introduction}

The ordinary chromatic number $\chi(G)$ of a graph $G$ (the minimum number of colors needed
to color the vertices so that adjacent vertices have distinct colors) is among the most studied
graph parameters, inspiring many variations and generalizations.  Among the most natural ones
is the \emph{fractional chromatic number} $\chi_f(G)$, obtained as the fractional relaxation of an integer linear
program defining the chromatic number.  As there are many equivalent ways how to define the fractional chromatic number~\cite{ScheinermanUllman2011},
let us choose one which is convenient with regards to the topic of this paper.

A \emph{weight assignment} for a graph $G$ is a function $w:V(G)\to\mathbb{R}^+_0$ which is not identically $0$.
For a function $f:X\to \mathbb{R}$ and a set $Y\subseteq X$, we define $f(Y)=\sum_{y\in Y} f(y)$.
Let $\alpha_w(G)$ denote the maximum weight of an independent set in $G$, that is,
$$\alpha_w(G)=\max\{w(Y):\text{$Y\subseteq V(G)$, $Y$ is independent in $G$}\}.$$
Then the fractional chromatic number $\chi_f(G)$ is defined as the supremum of $\tfrac{w(V(G))}{\alpha_w(G)}$
over all weight assignments $w$ for $G$.  Note that this can be expressed as a linear optimization problem,
and thus the supremum could be replaced by maximum in the definition.

The average weight of a color class in a proper coloring of $G$ by $\chi(G)$ colors is $\tfrac{w(V(G))}{\chi(G)}$,
showing that $\alpha_w(G)\ge \tfrac{w(V(G))}{\chi(G)}$ for every weight assignment $w$, and thus
$\chi_f(G)\le \chi(G)$.  The important question of whether the chromatic number
can be bounded by a function of the fractional chromatic number was answered in negative by Lov\'asz~\cite{kneser},
who proved that for all positive integers $a\ge 2b$, the Kneser graph $K_{a:b}$ has chromatic number exactly $a-2b+2$,
while it is known to have fractional chromatic number exactly $a/b$.  Consequently, the Kneser graphs $K_{(2b+c):b}$
have chromatic number $c+2$ (which can be arbitrarily large) and fractional chromatic number $2+c/b$
(which can at the same time be arbitrarily close to $2$, by choosing $b\gg c$).

From the other side, the fractional chromatic number is naturally lower bounded by the \emph{Hall ratio} $\rho(G)$
of the graph, defined as the maximum of $\tfrac{|V(H)|}{\alpha(H)}$ over all subgraphs $H$ of $G$,
where $\alpha(H)$ denotes the maximum size of an independent set of $H$.  Equivalently,
$\rho(G)$ is equal to the maximum of $\tfrac{w(V(G))}{\alpha_w(G)}$ over all $\{0,1\}$-valued weight assignments $w:V(G)\to \{0,1\}$,
which clearly implies $\chi_f(G)\ge \rho(G)$.  Let us remark that for Kneser graphs, the fractional chromatic number
and the Hall ratio coincide.  Furthermore, arguments to show that the (fractional) chromatic number of a particular
graph is large often proceed by lower bounding the Hall ratio.  This led Harris~\cite{harris} to conjecture that the fractional chromatic number
can be upper bounded by a function (actually even a linear function) of the Hall ratio; this question is also implicit in
the discussion at the end of Johnson Jr~\cite{hallneq}.  Although this is a rather unlikely proposition, it does not seem
easy to disprove.

Johnson Jr~\cite{hallneq} proved that there exist graphs $G$ for which $\chi_f(G)-\rho(G)$ is arbitrarily large.
Daneshgar et al.~\cite{dhj} and Barnett~\cite{barnett} constructed graphs with $\chi_f(G)\ge \tfrac{6}{5}\rho(G)$
and $\chi_f(G)\ge \tfrac{343}{282}\rho(G)$, respectively.  Note that if $G_k$ is obtained from $k$ vertex-disjoint
copies of $G$ by adding all edges between distinct copies, then $\chi_f(G_k)=k\chi_f(G)$ and $\rho(G_k)=k\rho(G)$,
and thus there also exist graphs with the same ratio $\chi_f(G)/\rho(G)$ and arbitrarily large Hall ratio.
Since the Hall ratio is lower bounded by the clique number of the graph, the possible counterexamples to Harris' conjecture
need to have a bounded clique number
and unbounded (fractional) chromatic number.  However, the probabilistic constructions of such graphs usually have
large Hall ratio.  Among the non-probabilistic constructions, iterated Mycielski graphs were investigated by
Cropper et al.~\cite{cropper2006hall}; they conclude that the Hall ratio of iterated Mycielski graphs is arbitrarily
large, and thus they cannot serve as a counterexample.

In a recent breakthrough, Blumenthal et al~\cite{nonlin} constructed graphs showing that not only the fractional chromatic
number is not linear in the Hall ratio, it actually cannot be bounded by a polynomial function of the Hall ratio.
Nevertheless, their graphs still have unbounded Hall ratio, leaving the possibility of the fractional chromatic number
being bounded by a fast growing function of the Hall ratio (see also~\cite[Question 16]{nonlin}).

We resolve this issue by proving two results of independent interest on properties of $1$-subdivisions appearing in graphs
of large average degree. The \emph{$1$-subdivision} of a graph $H$ is the bipartite graph obtained from $H$ by subdividing
each edge exactly once.  Equivalently, the $1$-subdivision of $H$ is isomorphic to the \emph{incidence graph} of $H$, that is,
the graph with the vertex set $V(H)\cup E(H)$ and $v\in V(H)$ being adjacent to $e\in E(H)$ if and only if the edge $e$ is incident with $v$.
A graph $G$ \emph{contains the $1$-subdivision of $H$} if the $1$-subdivision of $H$ is isomorphic to a subgraph of $G$.

Dvo{\v{r}}{\'a}k~\cite{Dvo2007,subdivchar} proved that for every $c$, every graph of sufficiently large average degree contains the $1$-subdivision
of a graph of chromatic number at least $c$.  Firstly, we prove that this statement also holds for the fractional chromatic number.
\begin{theorem}\label{thm-frac}
For every integer $c\ge 10$, every graph of average degree at least $256c^3$ contains the $1$-subdivision
of a graph of fractional chromatic number at least $c$.
\end{theorem}
Secondly, we prove that the statement \emph{does not} hold for the Hall ratio.
\begin{theorem}\label{thm-hallr}
For every integer $d\ge 1$, there exists a bipartite graph of average degree at least $d$
that does not contain the $1$-subdivision of any graph with Hall ratio greater than $18$.
\end{theorem}
Applying Theorem~\ref{thm-frac} to graphs obtained using Theorem~\ref{thm-hallr} for $d=256c^3$, we conclude that
the fractional chromatic number cannot be bounded by any function of the Hall ratio.
\begin{corollary}\label{cor-negat}
For every positive integer $c$, there exists a graph of fractional chromatic number at least $c$ and Hall ratio at most $18$.
\end{corollary}

In view of these results, what can we say about the relationship between the fractional chromatic number and the Hall ratio?
Firstly, let us note that for a graph $G$ with $n$ vertices, $\chi_f(G)\le \chi(G)=O(\rho(G)\log n)$.  Indeed, we can
obtain a proper coloring of $G$ by repeatedly extracting the largest independent set, reducing the size of the graph
by a factor smaller than or equal to $(1-1/\rho(G))$, so the graph becomes empty after $O(\rho(G)\log n)$ iterations.
\begin{problem}\label{prob-fn}
Determine the smallest function $g:\mathbb{Z}^+\to\mathbb{R}^+$ such that for every graph $G$,
$\chi_f(G)\le \rho(G)g(|V(G)|)$.
\end{problem}
The inspection of the proofs of Theorems~\ref{thm-frac} and \ref{thm-hallr}
shows that the number of vertices of graphs we obtain in Corollary~\ref{cor-negat}
is double exponential in their fractional chromatic number, and thus $g(n)=\Omega(\log\log n)$.

As we discussed before, the Hall ratio can be viewed as defined analogously to the fractional chromatic number, but
with the weight functions restricted to $\{0,1\}$-valued ones.  Does there exist a natural nontrivial family $\FF$ of weight functions
such that, defining $\rho_{\FF}(G)=\max_{w\in \FF} \tfrac{w(V(G))}{\alpha_w(G)}$, the fractional chromatic number of any graph $G$
can be bounded by a function of $\rho_{\FF}(G)$?  To lower-bound the fractional chromatic number of the graphs we (probabilistically)
construct in the proof of Theorem~\ref{thm-frac}, we give vertices weights which are a.a.s. up to a constant factor proportional to their
degrees.  Hence, the first natural problem to consider in this direction is as follows.
\begin{problem}
Do there for some constant $c>0$ exist graphs $G$ of arbitrarily large fractional chromatic number such that $\rho(G)\le c$ and
$\alpha_{\deg_H}(H)\ge |E(H)|/c$ for every $H\subseteq G$?
\end{problem}

Before we prove our results, we discuss their implications in the theory of bounded expansion.  Next, we prove Theorem~\ref{thm-frac} in Section~\ref{sec-fr}
and Theorem~\ref{thm-hallr} in Section~\ref{sec-hr}.

\section{Classes with Bounded Expansion}
Classes of bounded expansion have been introduced in~\cite{POMNI} as a
generalization of classes with excluded minors, which is based on the notion of
shallow minors introduced by Plotkin, Rao and Smith~\cite{shallow} (who in turn
attribute the idea to Leiserson and Toledo). A {\em shallow
minor} at depth $r$ of a graph $G$ is any graph that can be obtained from $G$ by 
deleting vertices and edges and contracting pairwise vertex-disjoint subgraphs,
each of radius at most $r$.
A class $\mathcal C$ has bounded expansion if there exists a function
$f_1:\mathbb N\rightarrow\mathbb R$, such that every shallow minor at depth $r$ of a graph in $\mathcal C$ has average degree at most
$f_1(r)$.
Quite a few characterizations of bounded expansion classes have been given,
which involve many of the classical graph invariants~\cite{Sparsity}. For
instance, denoting by $\mathcal C\shm r$ the class of all shallow minors at depth
$r$ of graphs in the class $\mathcal C$, 
$\mathrm{mad}(G)$ the maximum average degree of a subgraph of $G$, and
$\mathrm{col}(G)$ the smallest integer $k$ such that every subgraph of $G$ has minimum degree less than $k$,
we have the following characterization
of classes with bounded expansion.

\begin{theorem}[\cite{Sparsity}]\label{thm-bechar}
For a class of graphs $\mathcal C$ the following properties are equivalent:
\begin{enumerate}[(i)]
	\item\label{it:0} The class $\mathcal C$ has bounded expansion;
        \item\label{it:1} there exists a function $f_2:\mathbb N\rightarrow\mathbb N$ such that for every $r\in\mathbb N$ and every $G\in\mathcal C\shm r$ we have $\mathrm{mad}(G)\leq f_2(r)$;
	\item\label{it:2} there exists a function $f_3:\mathbb N\rightarrow\mathbb N$ such that for every $r\in\mathbb N$ and every $G\in\mathcal C\shm r$ we have $\mathrm{col}(G)\leq f_3(r)$;
	\item\label{it:3} there exists a function $f_4:\mathbb N\rightarrow\mathbb N$ such that for every $r\in\mathbb N$ and every $G\in\mathcal C\shm r$ we have $\chi(G)\leq f_4(r)$;
\end{enumerate}
\end{theorem}

By Theorem~\ref{thm-frac}, we get a characterization of bounded expansion in terms of the fractional chromatic number.
\begin{corollary}\label{cor-fbe}
A class $\mathcal C$ of graphs has bounded expansion if and only if
\begin{enumerate}[(i)]
\setcounter{enumi}{4}
\item\label{it:4} there exists a function $f_5:\mathbb N\rightarrow\mathbb N$ such that for every $r\in\mathbb N$ and every $G\in\mathcal C\shm r$ we have $\chi_f(G)\leq f_5(r)$.
\end{enumerate}
\end{corollary}
\begin{proof}
If $\mathcal{C}$ has bounded expansion, then $\chi_f(G)\le\chi(G)\le f_4(r)$ for every $r\in\mathbb N$ and every $G\in\mathcal C\shm r$ by part (iv) of Theorem~\ref{thm-bechar}.

Conversely, suppose that $\chi_f(G)\le f_5(r)$ for every $r\in\mathbb N$ and every $G\in\mathcal C\shm r$.
Let $f_1(r)\colonequals 256\bigl(\max(10,f_5(2r+1)+1)\bigl)^3$.
Consider any $r\in\mathbb N$ and any graph $G\in\mathcal C\shm r$.
If the $1$-subdivision of a graph $H$ appears in $G$, then $H\in \mathcal C\shm (2r+1)$, and thus $\chi_f(H)\le f_5(2r+1)$.
By Theorem~\ref{thm-frac}, every graph of average degree at least $f_1(r)$ contains the $1$-subdivision
of a graph of fractional chromatic number at least $f_5(2r+1)+1$; hence, we conclude $G$ has average degree less than $f_1(r)$.
Consequently, $\mathcal{C}$ has bounded expansion.
\end{proof}

The situation is less clear for the Hall ratio, that is for classes $\mathcal C$ such that the following property holds:
\begin{enumerate}[(i)]
\setcounter{enumi}{5}
	\item\label{it:5} there exists a function $f_6:\mathbb N\rightarrow\mathbb N$ such that for every $r\in\mathbb N$ and every $G\in\mathcal C\shm r$ we have $\rho(G)\leq f_6(r)$. 
\end{enumerate}
While every class with bounded expansion clearly satisfies (vi), due to Theorem~\ref{thm-hallr} the converse argument used in the proof of Corollary~\ref{cor-fbe} fails in the Hall ratio setting.
Nevertheless, we also do not know any example of a class with unbounded expansion whose shallow minors have bounded Hall ratio, leaving the following question open.

\begin{problem}\label{prob-vi}
Is it true that a class $\mathcal C$ has bounded expansion if and only if (vi) holds?
\end{problem}

More generally, a class $\mathcal C$ is \emph{nowhere dense}~\cite{npom-nd2} if there exists a function $g:\mathbb N\rightarrow\mathbb N$ such that for every $r\in\mathbb N$
and every $G\in\mathcal C\shm r$ we have $\omega(G)\leq g(r)$, where $\omega(G)$ denotes the maximum size of a clique in $G$. Note that all bounded expansion classes are nowhere dense but the converse does not hold,
as witnessed by the class of graphs having their maximum degree bounded by their girth~\cite{Sparsity}.
The present knowledge of how the usual density-related graph invariants characterize bounded expansion classes or nowhere dense classes is as follows:

\[
\underbrace{\mathrm{mad}(G)+1\geq \mathrm{col}(G)\geq \chi(G)\geq \chi_f(G)}_{\text{bounded expansion}}\geq \underbrace{\rho(G)}_{?}\geq \underbrace{\omega(G)}_{\text{nowhere dense}}
\]

In particular, every class with property \eqref{it:5} is nowhere dense.  However, the converse does not hold, as a consequence of the next proposition.
\begin{proposition}
	There exists a nowhere dense class $\mathcal C$ with unbounded Hall ratio.
\end{proposition}
\begin{proof}
Indeed, for prime $p,q$ with Legendre symbol $\legendre{p}{q}=1$ with $q$ sufficiently larger
than $p$ there exists a $(p+1)$-regular graph $X_{p,q}$ with girth at least $2\log_p q$ and independence number at most $2\sqrt{p}/(p+1)|X_{p,q}|$ (see for instance~\cite{nevsetvril2013combinatorial}).
It follows that there exist a non-decreasing function $F:\mathbb N\rightarrow \mathbb N$ and a sequence $(G_n)_{n\in\mathbb N}$
of graphs with ${\rm girth}(G_n)\rightarrow\infty$, $\Delta(G_n)< F({\rm girth}(G_n))$ and $|G_n|/\alpha(G_n)\rightarrow\infty$.
Let $\mathcal C=\{G_n\mid n\in\mathbb N\}$. From the first two properties of $G_n$ we get that $\mathcal C$ is nowhere dense
(indeed, if $6r+3<{\rm girth}(G_n)$, then a shallow minor $G$ of $G_n$ at depth $r$ is triangle-free and thus satisfies $\omega(G)\le 2$;
if $6r+3\ge {\rm girth}(G_n)$, then $G$ has maximum degree at most $\Delta^{r+1}(G_n)<F^{r+1}({\rm girth}(G_n))\le F^{r+1}(6r+3)$, and thus $\omega(G)\le F^{r+1}(6r+3)+1$).
However, the last property implies that $\rho(G_n)\geq |G_n|/\alpha(G_n)$ is unbounded on $\mathcal C$.
\end{proof}

Thus, the property (vi) either characterizes bounded expansion classes (if the answer to Problem~\ref{prob-vi} is positive), or
it is strictly sandwiched between the properties of bounded expansion and nowhere-density.

\section{Fractional chromatic number}\label{sec-fr}

We now turn our attention to $1$-subdivisions appearing in graphs with large average degree.
We use a standard probabilistic argument to prove the following lemma.  To this end, we employ
Chernoff's inequality:  Suppose $X$ is the sum of $n$ independent Bernoulli variables with mean $p$,
and let $\mu=\text{E}[X]=np$.
Then
\begin{equation}\label{eq-cher-large}
\text{Prob}[X\ge (1+\delta)\mu] \le \exp\Bigl(-\frac{\delta^2\mu}{2+\delta}\Bigr)
\end{equation}
and
\begin{equation}\label{eq-cher-small}
\text{Prob}[X\le (1-\delta)\mu] \le \exp\Bigl(-\frac{\delta^2\mu}{2}\Bigr)
\end{equation}
for every $\delta\ge 0$.

\begin{lemma}\label{lemma-semireg}
For all integers $q\ge 1$ and $a\ge 20$, every graph $G$ of average degree at least $32aq$ contains a bipartite subgraph $H$
with the bipartition $(A,B)$ satisfying $|A|=q|B|$ and every vertex of $A$ having degree exactly $a$.
\end{lemma}
\begin{proof}
For a uniformly random subset $A'\subseteq V(G)$, the expected number of edges of $G$ with one end in $A'$ and the other end in
$V(G)\setminus A'$ is $|E(G)|/2$; hence, $G$ has a spanning bipartite subgraph $G_1$ of average degree at least $16aq$.
Note that if $v\in V(G_1)$ has degree less than $8aq$, then $G_1-v$ has average degree more than $16aq$.
Hence, by repeatedly deleting vertices of degree less than $8aq$, we obtain a non-empty subgraph $G_2$ of minimum
degree at least $8aq$.

Let $(A_2,B_2)$ be the bipartition of $G_2$; without loss of generality, we can assume $|A_2|\ge |B_2|$.
It suffices to show that there exists a set $B\subseteq B_2$ such that at least $q|B|$ vertices of $A_2$
have each at least $a$ neighbors in $B$; then, we can set $A$ to be a set of $q|B|$ such vertices of $A_2$,
and we obtain $H$ from $G_2[A\cup B]$ by deleting all but $a$ edges incident with each vertex of $A$.
If $|A_2|\ge q|B_2|$, we can set $B=B_2$.  Hence, suppose $|A_2|<q|B_2|$.

Let $B\subseteq B_2$ be chosen by taking each element of $B_2$ independently at random with probability $p=\tfrac{|A_2|}{4q|B_2|}$.
Consider any vertex $v\in A_2$, and let $B_v$ be the number of neighbors of $v$ in $B$.  Let $\mu_v=\text{E}[B_v]$; we have
$$\mu_v=p\deg v\ge \frac{|A_2|}{4q|B_2|}\cdot 8aq=\frac{2a|A_2|}{|B_2|}\ge 2a\ge 40.$$
Consequently, by (\ref{eq-cher-small}),
$$\text{Prob}[B_v<a]\le \text{Prob}[B_v\le \mu_v/2]\le \exp(-\mu_v/8)\le e^{-5}.$$
Hence, Markov's inequality implies that with probability at least $1-2e^{-5}$, at least $|A_2|/2$ vertices of $A_2$
have at least $a$ neighbors in $B$.

The expected size of $B$ is $\mu=p|B_2|=\tfrac{|A_2|}{4q}$.  Since $G_2$ has minimum degree at least $8aq$,
we have $|A_2|\ge 8aq$, and thus $\mu\ge 2a\ge 40$.  By (\ref{eq-cher-large}),
$$\text{Prob}\bigl[|B|\ge \tfrac{|A_2|}{2q}\bigr]=\text{Prob}[|B|\ge 2\mu] \le \exp(-\mu/3)<e^{-13}.$$
Therefore, with probability at least $1-2e^{-5}-e^{-13}>0$,
$|B|<\tfrac{|A_2|}{2q}$ and at least $|A_2|/2>q|B|$ vertices of $A_2$ have at least $a$ neighbors in $B$, as required.
\end{proof}

For a graph $F$, let $\deg_F:V(G)\to\mathbb{Z}^+_0$ denote the function assigning to each vertex its degree in $F$.
Recall that for $Z\subseteq V(F)$, $\deg_F(Z)$ then denotes the sum of the degrees of vertices in $Z$.
Given a graph $H$ with the bipartition $(A,B)$, let $H_B$ denote the random graph with vertex set $B$ obtained
by, independently for each $v\in A$, choosing uniformly at random a pair of neighbors of $v$ and joining them by an edge.

\begin{lemma}\label{lemma-weight}
Let $a\ge 2$ and $q\ge 1$ be integers.
Let $H$ be a bipartite graph with the bipartition $(A,B)$ such that vertices of $A$ have degree exactly $a$
and $|A|=q|B|$.  Let $n=|B|$.
Let $Z\subseteq B$ be a set with $\deg_H(Z)\ge (\sqrt{q}a+q)n$.
Then the probability that $Z$ is an independent set in $H_B$ is less than $2^{-n}$.
\end{lemma}
\begin{proof}
For each vertex $v\in A$, let $d(v)$ denote the number of neighbors of $v$ in $Z$.
Then the probability $p$ that $Z$ is an independent set in $H_B$ is
\begin{align*}
p&=\prod_{v\in A}\Bigl(1-\frac{d(v)(d(v)-1)}{a(a-1)}\Bigr)\le \prod_{v\in A}\Bigl(1-\frac{d(v)(d(v)-1)}{a^2}\Bigr)\\
&\le\text{exp}\Bigl(-\frac{1}{a^2}\sum_{v\in A}d(v)(d(v)-1)\Bigr).
\end{align*}
Note that $\sum_{v\in A} d(v)=\deg_H(Z)$.  Using the well-known inequality between the quadratic and the arithmetic mean
$\sum_{i=1}^m x^2_i\ge \tfrac{1}{m}\Bigl(\sum_{i=1}^m x_i\Bigr)^2$,
we conclude
\begin{align*}
\sum_{v\in A}d(v)(d(v)-1)&=\sum_{v\in A}d^2(v)-\deg_H(Z)\ge \frac{\deg_H^2(Z)}{qn}-\deg_H(Z)\\
&=\frac{\deg_H(Z)(\deg_H(Z)-qn)}{qn}\ge \frac{(\sqrt{q}a+q)n\cdot \sqrt{q}an}{qn}\ge a^2n.
\end{align*}
Consequently,
$$p\le\text{exp}(-n)<2^{-n}.$$
\end{proof}

Since $B$ has only $2^n$ subsets, with non-zero probability no subset of large weight with respect to
the $\deg_H$ weight function is independent in $H_B$.
\begin{corollary}\label{cor-wind}
Let $a\ge 2$ and $q\ge 1$ be integers.
Let $H$ be a bipartite graph with the bipartition $(A,B)$ such that vertices of $A$ have degree exactly $a$
and $|A|=q|B|$.  With non-zero probability, $\deg_H(Z)<(\sqrt{q}a+q)|B|$ for every independent set $Z$ in $H_B$.
\end{corollary}

This clearly gives a lower bound on the fractional chromatic number of $H_B$.
\begin{corollary}\label{cor-frac}
Let $a\ge 2$ and $q\ge 1$ be integers.
Let $H$ be a bipartite graph with the bipartition $(A,B)$ such that vertices of $A$ have degree exactly $a$
and $|A|=q|B|$.
With non-zero probability, $\chi_f(H_B)>\frac{\deg_H(B)}{(\sqrt{q}a+q)|B|}=\frac{qa}{\sqrt{q}a+q}$.
\end{corollary}
\begin{proof}
With non-zero probability, $\deg_H(Z)<(\sqrt{q}a+q)|B|$ for every independent set $Z$ in $H_B$.
Hence, $\alpha_{\deg_H}(H_B)<(\sqrt{q}a+q)|B|$, and
$$\chi_f(H_B)\ge \frac{\deg_H(V(H_B))}{\alpha_{\deg_H}(H_B)}>\frac{\deg_H(B)}{(\sqrt{q}a+q)|B|}.$$
\end{proof}

We are now ready to prove our first main result.

\begin{proof}[Proof of Theorem~\ref{thm-frac}]
Suppose $G$ is a graph of average degree at least $256c^3$, where $c\ge 10$.  Letting $a=2c$ and $q=a^2=4c^2$,
the average degree of $G$ is at least $32aq$, and thus by Lemma~\ref{lemma-semireg}, $G$ 
contains a bipartite subgraph $H$ with the bipartition $(A,B)$ satisfying $|A|=q|B|$ and every vertex of $A$ having degree exactly $a$.
By Corollary~\ref{cor-frac}, with non-zero probability the graph $H_B$ has fractional chromatic number
greater than $\frac{qa}{\sqrt{q}a+q}=c$.
Note that the $1$-subdivision of $H_B$ is a subgraph of $H$, concluding the proof.
\end{proof}

\section{Hall ratio}\label{sec-hr}
Before proceeding to the proof of Theorem~\ref{thm-hallr}, let us provide some intuition.
If we could prove a variant of Lemma~\ref{lemma-semireg} where not only the vertices of $A$, but also the vertices
of $B$ were of (roughly) the same degree, then Corollary~\ref{cor-wind} would give a lower bound on the Hall ratio of $H_B$.
However, it is known that this is not possible; there are graphs which do not contain any such (nearly) regular subgraphs,
as shown by Pyber, R{\"o}dl, and Szemer{\'e}di~\cite{PRS}.  Hence, it is natural to consider the graphs with this property that they
constructed (with a slightly different choice of the parameters).

Let $M$ be a positive integer, let $\varepsilon_M=4^{-M-1}$, and let $n$ be the $4^M$-th power of an integer.
Let $G_{n,M}$
be the random graph whose vertex set consists of disjoint sets $A$, $B_1$, \ldots, $B_M$,
with $|A|=n$ and $|B_i|=n^{1-\varepsilon_M4^i}$ for $i=1,\ldots,M$,
with edge set obtained as follows. For each vertex $u\in A$ and $i=1,\ldots,M$,
choose a vertex $v\in B_i$ independently uniformly at random and add the
edge $uv$.  Let $B=B_1\cup\ldots\cup B_M$. Clearly, $G_{n,M}$ is bipartite
and $|E(G_{n,M})|=nM$. Furthermore, for $n$ sufficiently large, we have
$|B|\le n$, and thus the average degree of $G_{n,M}$ is at least $M$.
For $v\in B$, let $i(v)$ denote the index such that $v\in B_{i(v)}$.

Suppose the $1$-subdivision of a graph $H$ appears in another graph $G$.
We call the vertices of the $1$-subdivision corresponding to vertices of $H$
the \emph{branch vertices} and those corresponding to the edges of $H$ the \emph{subdivision vertices}.
For an integer $m\in\{2,\ldots, M\}$, let $\eB_m$ denote the
event that $G_{n,M}$ contains the $1$-subdivision of a graph $H$
of average degree at least $8$, such that $|V(H)|\le |B_m|$
and all the branch vertices of the $1$-subdivision are contained in
$B_1\cup \ldots \cup B_{m-1}$.

\begin{lemma}\label{lemma-notinb}
Let $M$ be a positive integer.  For a sufficiently large $4^M$-th power $n$,
the probability
that $\bigvee_{m=2}^M \eB_m$ holds in $G_{n,M}$ is less than $1/2$.
\end{lemma}
\begin{proof}
For integers $m\in\{2,\ldots, M\}$, $s$ such that $1\le s\le |B_m|$ and $t\ge 4s$,
let $\eB_{m,s,t}$ denote the event that $G_{n,M}$ contains the $1$-subdivision of a graph $H$
with $s$ vertices and $t$ edges, such that all the branch vertices of the $1$-subdivision are contained in 
$B_1\cup \ldots \cup B_{m-1}$.

Let us bound the probability of $\eB_{m,s,t}$.  Recall that for $n$ large enough, we have $|B|\le n$.
We can choose the branch vertices and subdivision vertices of $H$ in at most
\begin{equation}\label{eq:fbr}
\binom{n+|B_1\cup \ldots \cup B_{m-1}|}{s+t}\le \binom{2n}{s+t}\le (2e)^{2t}(n/s)^{s+t}
\end{equation}
ways.  For the selected subdivision vertices, we can choose two branch vertices to which they are adjacent
in at most
\begin{equation}\label{eq:fgr}
s^{2t}
\end{equation}
ways, thus determining the graph $H$.  Now, suppose that $z$ is a subdivision vertex
representing the edge $uv$ of $H$.  The probability that $G_{n,M}$ contains the edges $zu$ and $zv$
is $0$ if $i(u)=i(v)$, and $n^{\varepsilon_M(4^{i(u)}+4^{i(v)})-2}\le n^{2(\varepsilon_M4^{m-1}-1)}$ otherwise.
Consequently, the probability that the chosen $1$-subdivision of $H$ actually appears in $G_{n,M}$ as a subgraph is at most
\begin{equation}\label{eq:fpr}
n^{2t(\varepsilon_M4^{m-1}-1)}.
\end{equation}
Hence, using (\ref{eq:fbr}), (\ref{eq:fgr}), (\ref{eq:fpr}), and the assumption that $s\le |B_m|=n^{1-\varepsilon_M4^m}$, we have
\begin{align*}
\text{Prob}[\eB_{m,s,t}]&\le (2e)^{2t}(n/s)^{s+t}s^{2t}n^{2t(\varepsilon_M4^{m-1}-1)}\\
&=(2e)^{2t}(s/n)^{t-s}n^{2t\varepsilon_M4^{m-1}}\\
&\le (2e)^{2t}\Bigl(n^{\varepsilon_M4^{m-1}}\Bigr)^{4s-2t}.
\end{align*}
Since $t\ge 4s$, we conclude that for sufficiently large $n$, we have
$$\text{Prob}[\eB_{m,s,t}]\le (2e)^{2t}\Bigl(n^{\varepsilon_M4^{m-1}}\Bigr)^{-t}\le (8M)^{-t}.$$
Note that
$$\eB_m=\bigvee_{s\ge 1}\bigvee_{t\ge 4s}\eB_{m,s,t},$$
and thus
\begin{align*}
\text{Prob}\left[\bigvee_{m=2}^M\eB_m\right]&\le \sum_{m=2}^M\sum_{s\ge 1}\sum_{t\ge 4s} (8M)^{-t}\\
&\le 2\sum_{m=2}^M\sum_{s\ge 1}(8M)^{-4s}\\
&\le 4\sum_{m=2}^M (8M)^{-4}<1/2,
\end{align*}
as required.
\end{proof}

For an integer $m\in\{2,\ldots, M\}$, let $\eA_m$ denote the
event that $G_{n,M}$ contains the $1$-subdivision of a graph $H$
of average degree at least $8$, such that $|V(H)|\le |B_m|$
and all the subdivision vertices of the $1$-subdivision are contained in
$B_1\cup \ldots \cup B_{m-1}$.

\begin{lemma}\label{lemma-notina}
Let $M$ be a positive integer.  For a sufficiently large $4^M$-th power $n$, the probability
that $\bigvee_{m=2}^M \eA_m$ holds in $G_{n,M}$ is less than $1/2$.
\end{lemma}
\begin{proof}
For integers $m\in\{2,\ldots, M\}$, $s$ such that $1\le s\le |B_m|$ and $t\ge 4s$,
let $\eA_{m,s,t}$ denote the event that $G_{n,M}$ contains the $1$-subdivision of a graph $H$
with $s$ vertices and $t$ edges, such that all the subdivision vertices of the $1$-subdivision are contained in 
$B_1\cup \ldots \cup B_{m-1}$.

Let us bound the probability of $\eA_{m,s,t}$.
We can choose the branch vertices in at most
\begin{equation}\label{eq:bv}
\binom{n}{s}\le (en/s)^s
\end{equation} ways.  The edges of $H$ can be selected
in
\begin{equation}\label{eq:es}
\binom{\binom{s}{2}}{t}\le \binom{s^2}{t}\le (es^2/t)^t\le (es)^t
\end{equation}
ways.
For each edge $e$ of $H$, let $i(e)$ denote the index
such that the subdivision vertex representing $e$ belongs to $B_{i(e)}$; note that there are at most
\begin{equation}\label{eq:chi}
M^t
\end{equation}
functions $i:E(H)\to \{1,\ldots, M\}$.
For $i=1,\ldots,m-1$, let $Q_i=\{e\in E(H):i(e)=i\}$.  Since each vertex of $A$ has exactly one neighbor in $B_i$, observe
that $Q_i$ must be a matching.  The event $C_i$ that for every $e=uv\in Q_i$, the vertices $u$ and $v$ have a common neighbor in $B_i$,
has probability $|B_i|^{-|Q_i|}$ (once the neighbor $x$ of $u$ is selected, the probability that $v$ is also adjacent to $x$ is
$|B_i|^{-1}$, and since $Q_i$ is a matching, the events for distinct edges of $Q_i$ are independent).  The events $C_1$, \ldots, $C_{m-1}$
are independent, and thus the probability that $G_{n,M}$ contains the $1$-subdivision of $H$ whose subdivision vertices are in the prescribed
sets $B_1$, \ldots, $B_{m-1}$ is at most
\begin{equation}\label{eq:prb}
\prod_{i=1}^{m-1}|B_i|^{-|Q_i|}\le |B_{m-1}|^{-t}=n^{(\varepsilon_M4^{m-1}-1)t}.
\end{equation}
Hence, using (\ref{eq:bv})--(\ref{eq:prb}) and the assumptions that $s\le |B_m|=n^{1-\varepsilon_M4^m}$ and $t\ge 4s$, we have
\begin{align*}
\text{Prob}[\eA_{m,s,t}]&\le (en/s)^s(es)^tM^tn^{(\varepsilon_M4^{m-1}-1)t}\\
&\le (e^2M)^t (s/n)^{t-s}n^{\varepsilon_M4^{m-1}t}\\
&\le (e^2M)^tn^{\varepsilon_M4^{m-1}(4s-3t)}\\
&\le (e^2M)^t\Bigl(n^{\varepsilon_M4^{m-1}}\Bigr)^{-t}.
\end{align*}
We conclude that for sufficiently large $n$, we have $\text{Prob}[\eA_{m,s,t}]\le (8M)^{-t}$, and calculating
as at the end of the proof of Lemma~\ref{lemma-notinb}, we obtain
$$\text{Prob}\left[\bigvee_{m=2}^M\eA_m\right]<1/2.$$
\end{proof}

Combining Lemmas~\ref{lemma-notinb} and \ref{lemma-notina}, we obtain the following.

\begin{corollary}\label{cor-nM}
For any positive integer $M$ and a sufficiently large $4^M$-th power $n$, there exists a graph
$\widetilde{G}_{n,M}$ with vertex set consisting of disjoint sets $A$, $B_1$, \ldots, $B_m$,
with $|A|=n$ and $|B_i|=n^{1-\varepsilon_M4^i}$ for $i=1,\ldots,M$, such that each vertex of $A$
has exactly one neighbor in $B_i$ for $I=1,\ldots, M$, and neither $\eA_m$ nor $\eB_m$ holds
for any $m\in\{2,\ldots, M\}$.
\end{corollary}

By Tur\'an's theorem, we have the following.
\begin{lemma}\label{lemma-avgdeg}
Every graph $G$ has average degree at least $\tfrac{|V(G)|}{\alpha(G)}-1$.
\end{lemma}

Let us now argue about independent sets in $1$-subdivisions appearing in $\widetilde{G}_{n,M}$.
\begin{lemma}\label{lemma-setsinb}
For any positive integer $M$ and a sufficiently large $4^M$-th power $n$, if $H$ is a graph
whose $1$-subdivision appears in $\widetilde{G}_{n,M}$ with all branch vertices contained
in $B=B_1\cup\ldots\cup B_M$, then $\alpha(H)\ge |V(H)|/18$.
\end{lemma}
\begin{proof}
Let $B_{M+1}=\emptyset$.
Suppose for a contradiction that $\alpha(H)<|V(H)|/18$.  Let $m\in \{1,\ldots, M\}$ be the smallest integer
such that $|V(H)|\ge |B_{m+1}|$.  For sufficiently large $n$, we have
$|B_{m+2}\cup\ldots\cup B_M|\le |B_{m+1}|/4\le |V(H)|/4$.  Since every vertex of $A$ has exactly one neighbor
in $B_i$ for $i=1,\ldots, M$, observe that $|V(H)\cap B_i|$ is an independent set in $H$,
and thus $|V(H)\cap B_i|\le\alpha(H)<|V(H)|/18$.  Consequently, letting $H'=H-(B_m\cup\ldots\cup B_M)$,
we have
\begin{align*}
|V(H')|&=|V(H)|-|V(H)\cap (B_{m+2}\cup\ldots\cup B_M)|-|V(H)\cap (B_m\cup B_{m+1})|\\
&\ge |V(H)|-|V(H)|/4-2\cdot |V(H)|/18>|V(H)|/2,
\end{align*}
and in particular $V(H')\neq \emptyset$; it follows that $m\ge 2$.
By the minimality of $m$, it follows that $|V(H)|<|B_m|$.
Furthermore, $\alpha(H')\le \alpha(H)<|V(H)|/18\le |V(H')|/9$, and thus
the average degree of $H'$ is at least $8$ by Lemma~\ref{lemma-avgdeg}.  Therefore, $H'$ shows that $\widetilde{G}_{n,M}$ satisfies
$\eB_m$, which is a contradiction.
\end{proof}

\begin{lemma}\label{lemma-setsina}
For any positive integer $M$ and a sufficiently large $4^M$-th power $n$, if $H$ is a graph
whose $1$-subdivision appears in $\widetilde{G}_{n,M}$ with all branch vertices contained
in $A$, then $\alpha(H)\ge |V(H)|/13$.
\end{lemma}
\begin{proof}
Let $B_{M+1}=\emptyset$.
Suppose for a contradiction that $\alpha(H)<|V(H)|/13$.  Let $m\in \{1,\ldots, M\}$ be the smallest integer
such that $|V(H)|\ge |B_{m+1}|$.  For sufficiently large $n$, we have
$|B_{m+2}\cup\ldots\cup B_M|\le |B_{m+1}|\le |V(H)|$, and thus at most $|V(H)|$ edges of $H$
use subdivision vertices in $B_{m+2}\cup\ldots\cup B_M$.  Since every vertex of $A$ has exactly one neighbor
in $B_i$ for $i=1,\ldots, M$, recall that the edges of $H$ using subdivision vertices in $B_i$ form a matching,
and thus there are at most $|V(H)|/2$ of them.  Consequently, letting $H'$ be the spanning subgraph of $H$
containing only the edges using subdivision vertices in $B_1\cup\ldots\cup B_{m-1}$,
we have $|E(H')|\ge |E(H)|-2|V(H)|$.  By Lemma~\ref{lemma-avgdeg}, the average degree of $H$ is at least $12$,
and $|E(H)|\ge 6|V(H)|$; consequently, $|E(H')|\ge 4|V(H)|=4|V(H')|$.  In particular, $E(H')\neq\emptyset$,
and thus $m\ge 2$.  By the minimality of $m$, it follows that $|V(H')|=|V(H)|<|B_m|$.
Therefore, $H'$ shows that $\widetilde{G}_{n,M}$ satisfies
$\eA_m$, which is a contradiction.
\end{proof}

We are now ready to prove the main result of this section.

\begin{proof}[Proof of Theorem~\ref{thm-hallr}]
We take $G=\widetilde{G}_{n,d}$ for a sufficiently large $4^d$-th power $n$.
If $H$ is a graph whose $1$-subdivision is contained in $G$, then $H$ is the union of (possibly null) vertex-disjoint graphs $H_A$ and $H_B$
whose $1$-subdivisions appear in $G$ with branch vertices in $A$ and in $B$, respectively.
By Lemmas~\ref{lemma-setsinb} and \ref{lemma-setsina}, we have $\alpha(H_A)\ge |V(H_A)|/13$ and $\alpha(H_B)\ge |V(H_B)|/18$,
and thus $$\alpha(H)=\alpha(H_A)+\alpha(H_B)\ge |V(H)|/18.$$
As the same argument applies to every subgraph of $H$, we conclude that $\rho(H)\le 18$.
\end{proof}

\section*{Acknowledgments}

The results on Hall ratio were inspired by discussions taking place during the first 
Southwestern German Workshop on Graph Theory in Karlsruhe.  We would like to thank the
reviewers, whose suggestions helped us improve the presentation of the paper.

\bibliographystyle{siam}
\bibliography{no1sub.bib}

\end{document}